\documentclass[12pt]{article}

\usepackage{amssymb}%
\usepackage{amsmath}%
\usepackage{amsthm}%
\usepackage{textcomp}
\usepackage{xcolor}%
\usepackage[enableskew]{youngtab} 
\usepackage{ytableau}

\allowdisplaybreaks%

{\theoremstyle{plain}%
 \newtheorem{theorem}{Theorem}
 \newtheorem{corollary}{Corollary}

}
{\theoremstyle{remark}

}
{\theoremstyle{definition}

\newtheorem{example}{Example}
}

\begin{document}

\begin{center}
{\large Schur-hooks and Bernoulli number recurrences}
\end{center}

\begin{center}
{\textsc{John M. Campbell}} 

 \ 

\end{center}

\begin{abstract}
  Given an identity relating families of Schur and power sum symmetric functions, this may be thought of as encoding 
 representation-theoretic properties   according to how the $p$-to-$s$ transition matrices provide the irreducible character tables for 
  symmetric groups. The case of the Murnaghan--Nakayama  rule for cycles provides that $p_{n} = \sum_{i = 0}^{n-1} (-1)^i s_{(n-i, 
 1^{i})}$, and, since the power sum generator $p_{n}$ reduces to $\zeta(2n)$   for the Riemann zeta function $\zeta$ and for specialized  
  values of the indeterminates involved in the inverse limit construction of the algebra of symmetric  functions, this motivates both  
  combinatorial and number-theoretic applications related to the given case of the Murnaghan--Nakayama rule. In this  direction, since  
  every Schur-hook admits an expansion in terms of twofold products of elementary and complete homogeneous generators, we exploit  
 this   property for the same specialization that allows us to express $p_{n}$ with the Bernoulli number $B_{2n}$, using remarkable results  
  due to Hoffman on   multiple harmonic series. This motivates our bijective approach, through the use of sign-reversing involutions,  
 toward the determination of identities that   relate Schur-hooks and power sum symmetric functions and that we apply to obtain a new  
 recurrence for Bernoulli numbers.  
\end{abstract}


\vspace{0.1in}


\section{Introduction}
  One of the most important identities in both algebraic combinatorics and representation theory is given by how the transition matrices  
  for expanding the   power sum bases of the homogeneous components of the algebra $\textsf{Sym}$ of symmetric functions in terms of  
  Schur functions provide the irreducible   character tables for the representations of symmetric groups. Closely related to this is the  
  Murnaghan--Nakayama rule for expanding products of power   sum generators and Schur functions in terms of the $s$-basis. Power sum  
  generators, when viewed as formal power series according to the inverse limit   construction of $\textsf{Sym}$, may be expressed with  
  Bernoulli numbers for specified values for the indeterminates involved in this construction. So, for   the same assignment of values, by  
  expressing Schur functions arising from the Murnaghan--Nakayama rule in terms of Bernoulli numbers, we would obtain   both 
 number-theoretic and representation-theoretic interpretations, via the resultant relation among Bernoulli numbers and via the 
 irreducible character  entries of the $p$-to-$s$ transition matrices.  

 The sequence $(B_{n} : n \in \mathbb{N}_{0})$ of Bernoulli numbers is typically defined according to the Laurent series expansion whereby 
 $ \frac{x}{e^{x} - 1} = \sum_{i = 0}^{\infty} B_{i} \frac{x^{i}}{i!}$ for $|x| < 2 \pi$, with $(B_{n} : n \in \mathbb{N}_{0}) $ $ = $ $ \big( 1$, 
 $ -\frac{1}{2}$, $ \frac{1}{6}$, $ 0$, $ -\frac{1}{30}$, $ 0$, $\frac{1}{42}$, $ 0$, $ -\frac{1}{30}$, $ 0$, $ \frac{5}{66}$, $ \ldots \big)$. Apart from 
 how Bernoulli numbers provide a prominent area within number theory, the sequence of Bernoulli numbers has even been considered as being among the 
 most important number sequences in all of mathematics \cite{Merca2020}, with applications in special functions theory, real analysis, differential geometry, 
 numerical analysis, and many other areas. This motivates the development of interdisciplinary areas based on the use of combinatorial tools in the 
 determination of Bernoulli number identities. 

 The Riemannn zeta function is such that $\zeta(x) = \sum_{i = 0}^{\infty} \frac{1}{i^x}$ for $\Re(x) > 1$. One of the most fundamental properties of the 
 Bernoulli numbers is given by Euler's relation such that 
\begin{equation}\label{EulerzetaB}
 \zeta(2 n) = (-1)^{n + 1} \frac{ \left( 2 \pi \right)^{2 n} }{ 2 \left( 2 n \right)! } B_{2 n}, 
\end{equation}
 for positive integers $n$. Informally, by expressing the left-hand side of \eqref{EulerzetaB} with the power sum generator $p_{n} \in \textsf{Sym}$, 
 and by similarly expressing the $e$- and $h$-generators of $\textsf{Sym}$ using 
 multiple harmonic series identities due to 
 Hoffman \cite{Hoffman1992}, 
 then, given an identity relating the specified generators, this provides a corresponding identity involving Bernoulli numbers. This 
 approach has been applied by Merca \cite{Merca2016Asymptotics,Merca2020,Merca2018,Merca2019,Merca2017}, but it appears that this 
 approach has not been considered in relation to Schur functions, providing a main purpose of our paper. 

 The Murnaghan--Nakayama rule that is reviewed in Section \ref{sectionPreliminaries} does not seem to have previously been considered in relation to 
 Bernoulli numbers. In this direction, we apply a sign-reversing involution to prove a cancellation-free evaluation in the $p$-basis for the first moment 
 for a case of the Murnaghan--Nakayama rule for cycles. As suggested above, apart from the number-theoretic interest in our recurrence for Bernoulli 
 numbers obtained through a relation among Schur-hooks and power sum symmetric functions, this is of representation-theoretic interest, in view of the 
 consequent character relation we obtain according to 
\begin{equation}\label{pchars}
 p_{\mu} = \sum_{\lambda \vdash n} \chi^{\lambda}_{\mathfrak{S}_{n}}(\mu) \, s_{\lambda}. 
\end{equation} 

\section{Preliminaries}\label{sectionPreliminaries}
 We adopt the convention whereby symmetric polynomials and symmetric functions are over $\mathbb{Q}$, writing 
\begin{equation}\label{Symsuper}
 \textsf{Sym}^{(n)} = \mathbb{Q}\left[ x_{1}, x_{2}, \ldots, x_{n} \right]^{\mathfrak{S}_{n}} 
\end{equation}
 to denote the set of symmetric polynomials in $n$ independent indeterminates, and letting the action of the symmetric group $\mathfrak{S}_{n}$ be given 
 by permuting the variables in \eqref{Symsuper} \cite[p.\ 17]{Macdonald1995}. Writing $ \textsf{Sym}^{(n)}_{i} $ in place of the set of homogeneous 
 symmetric polynomials with degree $i$ in \eqref{Symsuper}, we obtain the graded ring structure such that $$ \textsf{Sym}^{(n)} = \bigoplus_{i \geq 0} 
 \textsf{Sym}^{(n)}_{i}. $$ We then define 
\begin{equation}\label{Symsub}
 \textsf{Sym}_{i} := \lim_{\substack{\longleftarrow\\ n }} \textsf{Sym}_{i}^{(n)}, 
\end{equation}
 referring to Macdonald's text for details as to the inverse limit involved in \eqref{Symsub}, which, for our purposes, relies on truncation morphisms from $ 
 \mathbb{Q}[x_{1}$, $ x_{2}$, $ \ldots$, $ x_{m}]$ to $\mathbb{Q}[x_{1}, x_{2}, \ldots, x_{n}]$ restricted to obtain corresponding morphisms from 
 $ \textsf{Sym}^{(m)}$ to $ \textsf{Sym}^{(n)}$ \cite[pp.\ 17--19]{Macdonald1995}. The inverse limit in \eqref{Symsub} then allows us to define 
 $\textsf{Sym}$ so that 
\begin{equation}\label{Symdefinition}
 \textsf{Sym} := \bigoplus_{i \geq 0} \textsf{Sym}_{i}. 
\end{equation}
 Building on the work of Merca \cite{Merca2020}, we intend to exploit the relation in \eqref{EulerzetaB}
 by setting the indeterminates involved in 
 the construction of $\textsf{Sym}$ so that $x_{j} = \frac{1}{j^2}$ for all indices $j$. 

 For $n \geq 0$, the elementary symmetric generator $e_{n} \in \textsf{Sym}$ is such that $e_{0} = 0$ and such that $$ e_{n} = \sum_{1 \leq i_{1} < 
 i_{2} < \cdots < i_{n}} x_{i_{1}} x_{i_{2}} \cdots x_{i_{n}} $$ for $n > 0$. By writing $e_{n} = e_{n}(x_{1}, x_{2}, \ldots)$, we are to make use of 
 the remarkable result due to Hoffman \cite[Corollary 2.3]{Hoffman1992} such that 
\begin{equation}\label{attributedHoffman}
 e_{n}\left( \frac{1}{1^2}, \frac{1}{2^2}, \ldots \right) = \frac{\pi^{2n}}{(2n + 1)!}. 
\end{equation}
 By setting $$ h_{n} = \sum_{1 \leq i_{1} \leq i_{2} \leq \cdots \leq i_{n}} x_{i_{1}} x_{i_{2}} \cdots x_{i_{n}} $$ as the $n^{\text{th}}$ complete 
 homogeneous generator of $\textsf{Sym}$, and by writing $h_{n} = h_{n}(x_{1}, x_{2}, \ldots)$, it is known that 
\begin{equation}\label{Mercarediscovered}
 h_{n}\left( \frac{1}{1^2}, \frac{1}{2^2}, \ldots \right) = (-1)^n \frac{\pi^{2n}}{(2n)!} 
 \left( 2 - 2^{2n} \right) B_{2n}, 
\end{equation}
 and a proof of this was given by Merca \cite{Merca2016Asymptotics} 
 and can also be obtained from a case of Theorem 2.1
 from the work of Hoffman \cite{Hoffman1992}. 
 By setting $p_{n} = \sum_{i \geq 1} x_{i}^{n}$ as the $n^{\text{th}}$ 
 power sum generator of $\textsf{Sym}$, and by writing 
 $p_{n} = p_{n}(x_{1}, x_{2}, \ldots)$, we find that the Euler identity in \eqref{EulerzetaB} 
 is such that
\begin{equation}\label{patsquares}
 p_{n}\left( \frac{1}{1^2}, \frac{1}{2^2}, \ldots \right) 
 = (-1)^{n+1} \frac{ \left( 2 \pi \right)^{2n} }{ 2 \left( 2 n \right)! } B_{2 n}, 
\end{equation}
 as noted in MacDonald's text \cite[pp.\ 33--34]{Macdonald1995}. 

 An integer partition is a finite tuple of positive integers. The length of an integer partition $\lambda$ is denoted with $\ell(\lambda)$ and refers to the 
 number of entries or parts of $\lambda$. The sequence of these parts may be denoted by writing $\lambda = (\lambda_{1}, \lambda_{2}, \ldots, 
 \lambda_{\ell(\lambda)})$. For the empty partition $()$, we write $e_{()} = h_{()} = p_{()} = 1$, and for a nonempty partition $\lambda$, we write 
 $e_{\lambda} = e_{\lambda_{1}} e_{\lambda_{2}} \cdots e_{\lambda_{\ell(\lambda)}}$ and $h_{\lambda} = h_{\lambda_{1}} h_{\lambda_{2}} \cdots 
 h_{\lambda_{\ell(\lambda)}}$ and $ p_{\lambda} = p_{\lambda_{1}} p_{\lambda_{2}} \cdots p_{\lambda_{\ell(\lambda)}}$. From the direct sum 
 decomposition in \eqref{Symdefinition}, by writing $\mathcal{P}$ in place of the set of integer partitions, we have that $\{ e_{\lambda} \}_{\lambda \in 
 \mathcal{P}}$ and $\{ h_{\lambda} \}_{\lambda \in \mathcal{P}}$ and $\{ p_{\lambda} \}_{\lambda \in \mathcal{P}}$ are all bases of $\textsf{Sym}$. 

 The Schur functions are often regarded as providing the most important basis of $\textsf{Sym}$, 
 in view of how Schur functions are of core importance within algebraic combinatorics. It is common to define
 the Schur function indexed by $\lambda \in \mathcal{P}$ according to the Jacobi--Trudi rule such that 
\begin{equation}\label{JacobiTrudi}
 s_{\lambda} = \text{det}\left( h_{\lambda_{i} - i + j} \right)_{1 \leq i, j \leq n}, 
\end{equation}
 letting it be understood that expressions of the form $h_{k}$ vanish for $k < 0$. It seems that identities as in \eqref{JacobiTrudi} have not previously 
 been considered in relation to Bernoulli number identities. 

 For a partition $\lambda$, the \emph{diagram} associated with $\lambda$ is an arrangement of cells formed from $\ell(\lambda)$ horizontal rows 
 whereby the $i^{\text{th}}$ such row, listed from top to bottom and for $i \in \{ 1, 2, \ldots, \ell(\lambda) \}$, consists of $\lambda_{i}$ cells, with the 
 rows left-justified. For partitions $\lambda$ and $\mu$ such that $\lambda_{i} \geq \mu_i$ for all possible indices $i$, the \emph{skew diagram} $\lambda / 
 \mu$ is obtained by aligning the diagrams of $\lambda$ and $\mu$ by the upper left cells of these diagrams and by removing any overlapping cells. A 
 \emph{rim hook} is a skew diagram that is edgewise connected and that contains no $2 \times 2 $ configuration of cells, as in the following. 
 $$\young(::~~,~~~,~)$$ 

 The \emph{Murnaghan--Nakayama rule} may be formulated via an expansion of the form 
\begin{equation}\label{fullMN}
 p_{r} s_{\lambda} = \sum (-1)^{\text{ht}(\mu / \lambda) + 1} s_{\mu}, 
\end{equation}
 where the sum in \eqref{fullMN} is over all $\mu$ such that $\mu / \lambda$ is a rim hook of size $r$, and where the height of a skew tableau refers 
 to the number of its rows. The special case of \eqref{fullMN} whereby $\lambda = ()$ may be referred to as the \emph{Murnaghan--Nakayama 
 rule for a cycle} \cite[p.\ 116]{CeccheriniSilbersteinScarabottiTolli2010}, noting that this base case may be applied to obtain combinatorial interpretations 
 for the $p$-to-$s$ transition matrices providing the irreducible character tables for symmetric groups. See also the work of Murnaghan 
 \cite{Murnaghan1937} and of Nakayama \cite{Nakayama1941one,Nakayama1941two}. 
 
 The case of \eqref{fullMN} whereby $\lambda$ is empty reduces to 
\begin{equation}\label{MNforcycles}
 p_{n} = \sum_{i = 0}^{n-1} (-1)^i s_{(n-i, 1^{i})}, 
\end{equation}
 and \eqref{MNforcycles} provides a key tool in our work, writing $(a, 1^{b}) = \big(a, \underbrace{1, 1, \ldots, 1}_{b} \big)$, with partitions of this 
 form being referred to as as \emph{hooks}. We may thus refer to a Schur function indexed by a hook as a \emph{Schur-hook}. 
 Schur-hooks play important roles in many areas of combinatorics and representation theory, in view, for example, of the representation-theoretic 
 significance of \eqref{MNforcycles}. This was considered in our past work on combinatorial objects we refer to as \emph{bipieri tableaux} 
 \cite{Campbell2016}, and an identity for Schur-hooks applied in this past work provides another key to our current work. This Schur-hook identity is 
 such that 
\begin{equation}\label{shooktohe}
 s_{(a, 1^{b})} = \sum_{i = 0}^{b} (-1)^{i} h_{a + i} e_{b - i}, 
\end{equation}
 and \eqref{shooktohe} may be proved using sign-reversing involutions 
 on bipieri tableaux \cite{Campbell2016}. 

\section{From Schur-hooks to Bernoulli numbers}
 Our technique for deriving Bernoulii number identities relies on identities relating Schur-hooks and elements of the $p$-basis of $\textsf{Sym}$, by 
 rewriting Schur-hooks involved according to \eqref{shooktohe}, and by then applying the Hoffman identities in \eqref{attributedHoffman} and 
 \eqref{Mercarediscovered}. As a way of illustrating our technique, as a natural place to start, we apply it to the identity allowing us to expanding 
 $p$-generators in terms of Schur-hooks, as follows. 

 From the Murnaghan--Nakayama rule for cycles, we rewrite the summand in \eqref{MNforcycles} according to the Schur-hook identity in 
 \eqref{shooktohe}, i.e., so that 
\begin{equation}\label{pgeneratordouble}
 p_{n} = \sum_{ {\substack{ 0 \leq i \leq n - 1 \\ 0 \leq j \leq i }}}
 (-1)^{i + j} h_{n - i + j} e_{i - j}. 
\end{equation}
 According to Hoffman's multiple harmonic series identities in \eqref{attributedHoffman} and \eqref{Mercarediscovered}, 
 together with the Bernoulli number identity in 
 \eqref{patsquares}, we find that \eqref{pgeneratordouble} implies that 
\begin{equation}\label{pdoubleimplies}
 (2n + 1) B_{2n} = \sum_{{\substack{ 0 \leq i \leq n \\ 0 \leq j \leq i }}}
 \binom{2n+1}{2i-2j+1} B_{2n-2i+2j} \left( 2^{1 - 2 i + 2j} - 2^{2 - 2n} \right). 
\end{equation}
 Summing over $i \in \{ 0, 1, \ldots, n \}$ and, for each such index, over $j \in \{ 0, 1, \ldots, i \}$, 
 and then reversing the order of summation, we may obtain from \eqref{pdoubleimplies} an equivalent version of 
\begin{equation}\label{reverseequivalent}
 B_{2n} 2^{2n-1} = \sum_{j=0}^{n} \binom{2n}{2j-1} B_{2j} \left( 2^{2j-1} - 1 \right), 
\end{equation}
 with \eqref{reverseequivalent} providing a natural companion to 
 the identity due to Ramanujan \cite{Ramanujan1911} such that 
\begin{equation}\label{mainRamanujan}
 2 n + 1 = \sum_{j = 0}^{n} \binom{2 n + 1}{2 j} B_{2j} 2^{2j}. 
\end{equation}
 While the identity in \eqref{reverseequivalent} is related to the lacunary recurrences pioneered
 by Lehmer \cite{Lehmer1935} and can be obtained 
 using known identities involving Bernoulli numbers 
 (see \cite[Eq.\ (50.5.32)]{Hansen1975}) 
 the research interest 
 in Ramanujan's formula in \eqref{mainRamanujan} motivates how 
 our symmetric functions-based technique can be used to obtain new and further Bernoulli sum identities. 
 In this direction, \emph{nested} sums involving Bernoulli numbers, 
 that cannot be reduced in any obvious way (compared with \eqref{pdoubleimplies}) 
 are of a much more elusive nature, 
 and this motivates how we apply our technique
 to obtain a nested Bernoulli sum identity from 
 Theorem \ref{maintheorem} below. 

 If we consider the summand of the alternating sum in the special case of the Murnaghan--Nakayama rule in \eqref{MNforcycles}, as a natural way 
 of extending this case, we consider the first moment associated with the specified summand, 
 and this has led us to experimentally discover, 
 with the use of the {\tt SageMath} system, the following symmetric function
 identity that appears to be new and we are to 
 prove bijectively. 

\begin{theorem}\label{maintheorem}
 The relation 
\begin{equation}\label{displayintheorem}
 \sum_{i=0}^{2n} (-1)^i (n-i) s_{\left( 2 n + 1 - i, i \right)} = \sum_{i=1}^{n} p_{(2n+1-i,i)}. 
\end{equation}
 holds for natural numbers $n$. 
\end{theorem}

 Apart from the representation-theoretic properties that can be gleaned using Theorem \ref{maintheorem} according to the irreducible character 
 identity in \eqref{pchars}, Hoffman's multiple harmonic series identities in \eqref{attributedHoffman} and \eqref{Mercarediscovered}, can be 
 used, via \eqref{shooktohe}, to obtain from Theorem \ref{maintheorem} the following. 

\begin{corollary}\label{maincorollary}
 The relation 
 $$ \sum_{j=1}^{2n+1} \binom{4n+2}{2j-1} (2n+1 - j) 
 (1-2^{2j-1}) B_{2j} 
 = (4n+3) 2^{4n} B_{4n+2} $$
 holds for nonnegative integers $n$. 
\end{corollary}

\begin{proof}
 By Hoffman's multiple harmonic series relations applied to Theorem \ref{maintheorem}, 
 we obtain  that    the relation 
\begin{multline*}
 \sum_{\substack{ 0 \leq i \leq 2n \\ 0 \leq j \leq i }} 
 \binom{4n+2}{2j} \frac{n - i}{2j+1} \left( 2^{1-4n} - 2^{2-2j} \right) B_{4n-2j+2} 
 = \\ \sum_{i=1}^{n} \binom{4n+2}{2i} B_{2i} B_{4n-2i+2} 
\end{multline*}
 holds for natural numbers $n$. The right-hand side of this identity is half of the full sum from $i = 1$ to $2n$. Then, from the 
 convolution identity $$ \sum_{j=1}^{n-1} \binom{2n}{2j} B_{2j} B_{2n-2j} = -(2n+1) B_{2n} $$ attributed to Euler, we obtain  
  that  the          right-hand side of the above consequence of Theorem \ref{maintheorem} is such that $$ 
 \sum_{i=1}^{n} \binom{4n+2}{2i} B_{2i} B_{4n-2i+2} = -\frac{1}{2} (4n+3) B_{4n+2}. $$ We denote the left-hand side of the above 
 consequence of Theorem \ref{maintheorem} as $S$. We then rewrite $S$ so that 
 $$ S = \frac{1}{2^{4n}} \sum_{j=0}^{2n} \left( \sum_{i=j}^{2n} (n-i) \right) 
 \binom{4n+2}{2j} \frac{1}{2j+1} \big( 2 - 2^{4n-2j+2} \big) B_{4n-2j+2}. $$
 Evaluating the inner sum
 allows us to write
 $$ S = \frac{-1}{2^{4n+2}} \sum_{j=1}^{2n+1} \binom{4n+2}{2j-1} (2n+1-j) \big( 2 - 2^{2j} \big) B_{2j}, $$
 giving us an equivalent version of the desired result. 
\end{proof}

 The new recursion for $(B_{n} : n \in \mathbb{N}_{0})$ 
 highlighted in Corollary \ref{maincorollary} 
 is motivated by how there is a rich history about Bernoulli number recursions. 
 In this direction, an especially celebrated 
 Bernoulli sum identity is due to Miki \cite{Miki1978} and is such that 
\begin{equation}\label{celebratedMiki}
 \sum_{k=2}^{n-2} \frac{ B_{k} B_{n-k} }{ k (n-k) } - \sum_{k=2}^{n-2} \binom{n}{k} \frac{B_{k} B_{n-k}}{k (n-k)} 
 = \frac{2 B_{n} H_{n}}{n}, 
\end{equation}
 writing $H_{n} = 1 + \frac{1}{2} + \cdots + \frac{1}{n}$ to denote the $n^{\text{th}}$ harmonic number. 
 The relation in \eqref{celebratedMiki} was proved by Miki 
 via the Fermat quotient $(a^p - 1) /p$ mod $p^2$, 
 and an inequivalent proof via $p$-adic analysis is due to 
 Shiratani and Yokoyama \cite{ShirataniYokoyama1982}, 
 and this touches on 
 the number-theoretic interest in Corollary \ref{maincorollary}. 
 See also Gessel's alternative proof of Miki's identity 
 \cite{Gessel2005} along with many further references related to Miki's identity. 
 The goal of Section \ref{sectionbijective} below is to bijectively 
 prove Theorem \ref{maintheorem} and thus Corollary \ref{maincorollary}. 
 
\section{A bijective approach}\label{sectionbijective}
 Algebraic and combinatorial properties of Schur 
 functions are often revealed by determining cancellation-free formulas from alternating sums, as explored by 
 van Leeuwen \cite{vanLeeuwen200406}. This leads us to apply a sign-reversing involution to prove 
 the Schur-hook identity in Theorem \ref{maintheorem} and the consequent Bernoulli number identity in Corollary \ref{maincorollary}. 

 \ 

\noindent \emph{Proof of Theorem \ref{maintheorem}:} We rewrite the 
 right-hand side of \eqref{displayintheorem} to apply the Murnaghan--Nayakama rule, with 
\begin{align*}
 \sum_{i=1}^{n} p_{i} p_{2n+1-i} 
 & = \sum_{i=1}^{n} \left( \sum_{j = 0}^{i-1} (-1)^j s_{(i-j, 1^{j})} \right) p_{2n+1-i} \\ 
 & = \sum_{\substack{ 1 \leq i \leq n \\ 0 \leq j \leq i - 1 } } (-1)^j s_{(i-j, 1^{j})} p_{2n+1-i}. 
\end{align*}
 By the Murnaghan--Nakayama rule, we thus have that 
\begin{equation}\label{prooffirstlabel}
 \sum_{i=1}^{n} p_{i} p_{2n+1-i} 
 = \sum_{\substack{ 1 \leq i \leq n \\ 0 \leq j \leq i - 1 \\ 
 \text{$(2n+1-i)$-rim hooks $\mu / (i-j, 1^{j})$} } } 
 (-1)^{j + \text{ht}(\mu / (i-j, 1^{j})) + 1} s_{\mu}
\end{equation}
 By then letting $\mathcal{S}_{n}$ denote the set of all rim hooks
 of the form $\mu / (i-j, 1^{j})$ such that $1 \leq i \leq n$
 and $0 \leq j \leq i-1$, we define $\varphi_{n}\colon \mathcal{S}_{n} \to \mathcal{S}_{n}$ 
 as below. 
 For a rim hook 
 $\mu / (i-j, 1^{j})$, we denote the inner shape $(i-j, 1^{j})$
 with blank cells, and we denote any remaining cells in 
 $\mu / (i-j, 1^{j})$ with colored cells. 
 For a rim hook of the specified form, there is always at least one blank cell in the upper left, 
 subject to the given constraints on the indices $i$ and $j$. 

 If the most lower-left colored cell in a $(2 n + 1 - i)$-rim hook 
 of the form $\mu / (i-j, 1^{j})$ is immediately to the right of a blank cell (in the first column), 
 then we color this blank cell and any cells beneath it. 
 From the inner hook shape and the given constraints on the indices, 
 if this first condition holds, then it cannot also be the case that the most upper-right colored cell
 is immediately beneath a blank cell (in the first row). 
 In this second case, we color this blank cell and and any cells to the right of it. 
 In either case, the sign of $ (-1)^{j + \text{ht}(\mu / (i-j, 1^{j})) + 1}$ is reversed. 
 Given a rim hook that satisfies the conditions for either case, we let $\varphi_{n}$ 
 map this rim hook according to the specified procedures, respectively. 

 Suppose that the two cases given in the preceding paragraph are not satisfied. Again, there is at least one
 blank cell in the upper left, and we obtain a (possibly empty) rim hook (not necessarily appearing in $\mathcal{S}_{n}$) 
 by switching any colored cell in the first column to a blank cell, or by switching any colored cell in the 
 first row to a blank cell. If there is at least one colored cell in the first row and at least one colored cell in the first column, 
 then we perform the specified switching operation to the row/column with the least number of colored cells, noting that 
 a ``tie'' is not possible 
 by the parity of $2n+1$, 
 and we define $\varphi_{n}$ accordingly, again if the first two cases are not satisfied. 
 Otherwise, if the three preceding cases are not satisfied, we let 
 $\varphi_{n}$ map a rim hook in $\mathcal{S}_{n}$ to itself. 
 The given constraints on $i$, $j$, and the size of
 $(2n+1-i)$-rim hooks of the form $\mu / (i - j, 1^{j})$ then give us, from the four possible cases, 
 that $\varphi_{n}$ is an involution. 
 So, since $\varphi_{n}$ is an involution and since the sign $ (-1)^{j + \text{ht}(\mu / (i-j, 1^{j})) + 1}$ 
 is reversed for the first two cases, the sign is reversed in the third case. 

 A $(2n+1-i)$-rim hook in $\mathcal{S}_{n}$ 
 of shape $\mu / (i - j, 1^{j})$ 
 is mapped to itself by $\varphi_{n}$ if and only if $\mu$ is of hook shape, since
 if we were to attempt to apply the switching procedures given above, we would 
 again obtain a rim hook of outer hook shape, 
 but the sign would not be reversed for 
 the same outer hook shape. 
 So, from our sign-reversing involution, the right-hand side of \eqref{prooffirstlabel} 
 reduces to a signed sum of Schur-hooks, 
 with any two Schur-hooks of the same shape being of the same sign, i.e., the sum 
\begin{equation}\label{repeatedmoment}
 \sum_{i=0}^{2n} (-1)^i (n-i) s_{\left( 2 n + 1 - i, i \right)}, 
\end{equation}
 where the multiplicity given by the factor $(n-i)$ in the summand in \eqref{repeatedmoment} 
 is given by the $n - i$ choices given by the possible number of blank cells in the first row or column, 
 depending on whether the inner hook shape $(i - j, 1^{j})$ is vertical or horizontal. \qed 

\begin{example}
 For the $n = 3$ case of our proof of Theorem \ref{maintheorem}, 
 the matchings or pairings we obtain according to the 
 sign-reversing involution $\varphi_{n}$ are illustrated below. 

$$ - 
 s_{ \ytableausetup{smalltableaux} 
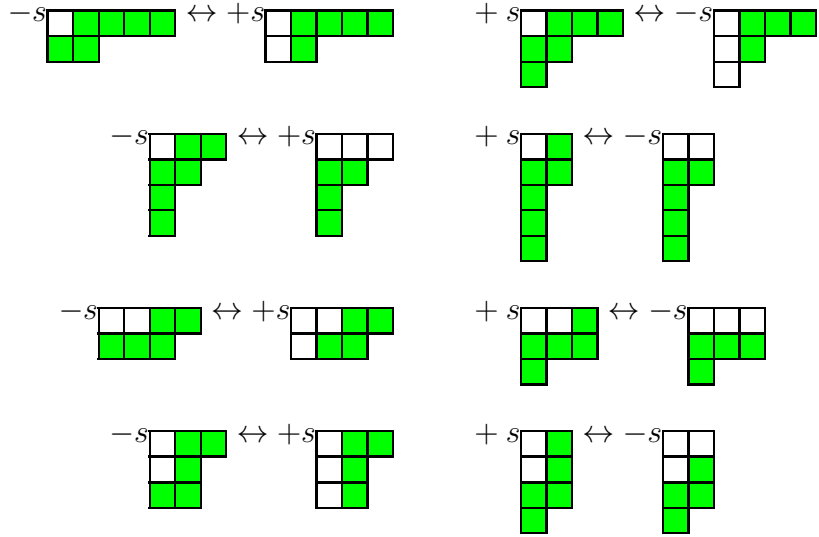
\begin{ytableau}
*(white) \null & *(green) \null & *(green) \null & *(green) \null & *(green) \null \\
 *(green) \null & *(green) \null 
 \end{ytableau} } \leftrightarrow + s_{ \begin{ytableau} 
 *(white) \null & *(green) \null & *(green) \null & *(green) \null & *(green) \null \\ 
 *(white) \null & *(green) \null 
 \end{ytableau} } 
 \ \ \ \ \ \ \ 
 + s_{ \begin{ytableau}
*(white) \null & *(green) \null & *(green) \null & *(green) \null \\ 
 *(green) \null & *(green) \null \\ 
 *(green) \null 
\end{ytableau} } \leftrightarrow - s_{ \begin{ytableau} 
 *(white) \null & *(green) \null & *(green) \null & *(green) \null \\ 
 *(white) \null & *(green) \null \\ 
 *(white) \null 
 \end{ytableau} } $$

$$ - s_{ \begin{ytableau}
*(white) \null & *(green) \null & *(green) \null \\
 *(green) \null & *(green) \null \\ 
 *(green) \null \\ 
 *(green) \null 
\end{ytableau} } \leftrightarrow + s_{ \begin{ytableau} 
 *(white) \null & *(white) \null & *(white) \null \\ 
 *(green) \null & *(green) \null \\ 
 *(green) \null \\ 
 *(green) \null 
 \end{ytableau} } 
 \ \ \ \ \ \ \ 
 + s_{ \begin{ytableau}
*(white) \null & *(green) \null \\
 *(green) \null & *(green) \null \\ 
 *(green) \null \\ 
 *(green) \null \\ 
 *(green) \null 
\end{ytableau} } \leftrightarrow - s_{ \begin{ytableau} 
 *(white) \null & *(white) \null \\ 
 *(green) \null & *(green) \null \\ 
 *(green) \null \\ 
 *(green) \null \\ 
 *(green) \null 
 \end{ytableau} } $$

$$ - s_{ \begin{ytableau} 
 *(white) \null & *(white) \null & *(green) \null & *(green) \null \\
 *(green) \null & *(green) \null & *(green) \null 
 \end{ytableau} } 
 \leftrightarrow + 
 s_{ \begin{ytableau} 
 *(white) \null & *(white) \null & *(green) \null & *(green) \null \\ 
 *(white) \null & *(green) \null & *(green) \null 
 \end{ytableau} } 
 \ \ \ \ \ \ \ 
 + s_{ \begin{ytableau} 
 *(white) \null & *(white) \null & *(green) \null \\
 *(green) \null & *(green) \null & *(green) \null \\ 
 *(green) \null 
 \end{ytableau} } \leftrightarrow - s_{ \begin{ytableau} 
 *(white) \null & *(white) \null & *(white) \null \\ 
 *(green) \null & *(green) \null & *(green) \null \\ 
 *(green) \null 
 \end{ytableau} } 
$$

$$ - s_{ \begin{ytableau} 
 *(white) \null & *(green) \null & *(green) \null \\ 
 *(white) \null & *(green) \null \\ 
 *(green) \null & *(green) \null 
 \end{ytableau} } \leftrightarrow + s_{ \begin{ytableau} 
 *(white) \null & *(green) \null & *(green) \null \\ 
 *(white) \null & *(green) \null \\ 
 *(white) \null & *(green) \null 
 \end{ytableau} } 
 \ \ \ \ \ \ \ 
 + s_{ \begin{ytableau} 
 *(white) \null & *(green) \null \\ 
 *(white) \null & *(green) \null \\ 
 *(green) \null & *(green) \null \\ 
 *(green) \null 
 \end{ytableau} } \leftrightarrow
 - s_{ \begin{ytableau} 
 *(white) \null & *(white) \null \\ 
 *(white) \null & *(green) \null \\ 
 *(green) \null & *(green) \null \\ 
 *(green) \null 
 \end{ytableau} } $$
\end{example}

\section{Conclusion}
 We briefly conclude with some further areas to explore. 

 We encourage further explorations based on the derivation of Bernoulli number identities
 from relations among Schur-hooks and $p$-basis elements. 
 For example, we have discovered that 
 $$ s_{\left( m, 2, 1^{2n+1} \right)} 
 = \sum_{i=1}^{2n+1} (-1)^{i+1} e_{2n-i+2} \sum_{j=1}^{i} s_{\left( m + j - 1, 1^{i-j+2} \right)} - \sum_{i=0}^{2n+1}
 s_{\left( m + i, 1^{2n-i+3} \right)}, $$ 
 and this provides a new triple sum identity for Bernoulli numbers. 

 Instead of using the Murnaghan--Nakayama rule, one could instead consider making use of the Littlewood--Richardson rule, 
 toward the goal of introducing bijective proofs of Bernoulli number identities as in Corollary \ref{maincorollary}. 
 For example, we may rewrite the right-hand sum in \eqref{displayintheorem} using Littlewood--Richardson 
 coefficients so that 
 $$ \sum_{i = 1}^{n} p_{\left( 2 n + 1 - i, i \right) } = 
 \sum_{\substack{ 1 \leq i \leq n \\ 0 \leq j \leq 2 n - i \\ 0 \leq k \leq i - 1 
 \\ \lambda }} 
 (-1)^{j + k} 
 c_{(2n+1-i - j, 1^{j}) \, (i - k, 1^{k})}^{\lambda} s_{\lambda}, 
 $$ 
 which suggests that a bijective proof of Corollary \ref{maincorollary}
 on Littlewood--Richardson tableaux may be possible. 

 Instead of making use of Hoffman's results for 
 $h_{n}\big( \frac{1}{1^2}$, $ \frac{1}{2^2}$, $ \ldots \big)$
 and $e_{n}\big( \frac{1}{1^2}$, $ \frac{1}{2^2}$, $ \ldots \big)$, 
 how could we instead apply, via the use of Schur-hooks, variants of these results 
 for expressions such as 
 $h_{n}\big( \frac{1}{2^2}$, $ \frac{1}{4^2}$, $ \ldots \big)$ 
 and $e_{n}\big( \frac{1}{2^2}$, $ \frac{1}{4^2}$, $ \ldots \big)$, or for alternating variants of such expressions? 

\subsection*{Acknowledgements}
 The author is grateful to acknowledge support from a Killam Postdoctoral Fellowship from the Killam Trusts, 
 and the author is very grateful to Karl Dilcher and to Christophe Vignat for very useful feedback concerning the author's discoveries.

 \

\noindent {\textsc{John M. Campbell}} 

\vspace{0.1in}

\noindent Department of Mathematics and Statistics

\noindent Dalhousie University 

\noindent Halifax, Nova Scotia, B3H 4R2, Canada

\vspace{0.1in}

\noindent {\tt jmaxwellcampbell@gmail.com}

\end{document}